\documentclass[11pt,reqno]{amsart}
\usepackage{amsmath,amssymb,amsthm,graphicx,psfrag}
\usepackage[usenames]{color}
\usepackage{hyperref}

\usepackage{latexsym}
\usepackage[cp850]{inputenc}
\usepackage{epsfig}
\usepackage{amscd}
\usepackage{amsfonts}

\usepackage[all]{xy}

\newtheorem{thm}{Theorem}               \newtheorem{lem}[thm]{Lemma}

\newtheorem{cor}[thm]{Corollary}
\newtheorem{prop}[thm]{Proposition}
                
\newtheorem*{defi*}{Definition} \newtheorem*{ex*}{Example}
\newtheorem*{thm*}{Theorem}         \newtheorem*{cor*}{Corollary}
\newtheorem*{rmk*}{Remark}          
     \newtheorem*{prop*}{Proposition}
\newtheorem*{conj*}{Conjecture}
\newtheorem{remark}[thm]{Remark}
\newtheorem{definition}[thm]{Definition}

\numberwithin{thm}{section}

\newtheorem*{namedtheorem}{\theoremname}
\newcommand{\theoremname}{testing}

\theoremstyle{remark}

\newcommand{\BR}{\mathbb R}         
         
         \newcommand{\BZ}{\mathbb Z}

\newcommand{\lra}{\longrightarrow}

\DeclareMathOperator{\Aut}{Aut} \DeclareMathOperator{\IA}{IA}
\DeclareMathOperator{\PS}{P\Sigma}      
\DeclareMathOperator{\GL}{GL}       

\begin{document}

\title[On the Andreadakis-Johnson filtration of $\Aut F_n$]{On the Andreadakis-Johnson filtration of \\ the automorphism  group of \\ a free group}

\author[F.~R.~Cohen]{F.~R.~Cohen}
\address{Dept.~of Mathematics \\
  University of Rochester \\
  Rochester, NY}
\email{cohf@math.rochester.edu}

\author[A.~Heap]{Aaron Heap}
\address{Dept.~of Mathematics \\
  SUNY Geneseo \\
  Geneseo, NY}
\email{heap@geneseo.edu}

\author[A.~Pettet]{Alexandra Pettet}
\address{Department of Mathematics \\
  University of Michigan \\
  Ann Arbor, MI}
\email{apettet@umich.edu}

\date{\today}

\maketitle

\begin{abstract}

Families of non-trivial cohomology classes are given for the
discrete groups belonging to the Johnson filtration of the
automorphism group of a  free group generated by $n$ letters. The
methods are (1) to analyze analogous classes for filtrations of a
subgroup of the pure symmetric automorphism group of a free group
and (2) to analyze features of these classes which are preserved by
the Johnson homomorphism with values in the Lie algebra of
derivations of a free Lie algebra. One consequence is that the ranks
of the cohomology groups in any fixed degree $i$ for $ 1 \leq i \leq
n-2$ for the Johnson filtrations of $IA_n$ increase without bound.
The actual classes constructed are fragile in the sense that they
vanish after passage to successive filtrations; furthermore, these
classes are all naturally in the image of the Johnson homomorphism
on the level of cohomology. The methods are similar to those
occurring within the theory of arrangements.

\

The authors congratulate Corrado De Concini on the occasion of his
60-th birthday.  One of the authors would like to thank Corrado for
a wonderful time arising from a discussion about cohomology of
groups with friends in the basement of Il Palazzone !
\end{abstract}

%
%

\section{Introduction and preliminaries}\label{sec:introduction}

Let $F_n = F[ x_1,\ldots,x_n]$ be the free group on generators $x_1,
\ldots, x_n$. There is a surjective natural map $Aut(F_n)
\rightarrow \GL_n (\BZ)$ which sends an automorphism to the induced
map on $H_1(F_n)$. The kernel $IA_n$ consists of exactly those
automorphisms which induce the identity on $H_1(F_n)$. This is the
$Aut(F_n)$-analogue of the {\it Torelli group}, the subgroup of the
mapping class group of a surface which acts trivially on the first
homology group of the surface.

The group $IA_n$ is the first in a series of subgroups belonging to
the {\it Andreadakis-Johnson filtration of $Aut(F_n)$}; that is,
the descending series $IA_n = J_n^1 \supset J_n^2 \supset \cdots$,
where $J_n^s$ consists of those automorphisms of $Aut(F_n)$ which
induce the identity map on the level of $F_n/\Gamma_n^{s+1}$; here
$\Gamma_n^{s+1}$ denotes the $(s+1)$-st stage of the descending central
series for $F_n$, as given in more detail below. The subgroup $J_n^s$
is sometimes called the $s$-th Andreadakis-Johnson filtration, or
$s$-th Johnson filtration below, which gives a decreasing filtration
of $IA_n$. The main goal in this paper is to establish quantitative
information about the cohomology of these groups. One implication of
the work here is that the ranks of the cohomology groups in any
fixed degree $i$ for $1 \leq i \leq n-2$ increase without bound by
going deeper into the Andreadakis-Johnson filtration.

A brief summary is given next for what was known previously about
the (co)homology of the Johnson filtrations.

In 1934, Magnus \cite{Magnus} provided a finite generating set for
$IA_n$ (see Section \ref{sec:Properties of McCool's group} below). The first homology and cohomology groups were partially
computed by Andreadakis \cite{Andreadakis}; completely by Kawazumi
in \cite{Kawazumi}. A number of degree two cohomology classes in the
image of cup product were computed by Pettet \cite{Pettet}. However
it is still unknown whether or not $IA_n$ is finitely presentable
for $n \geq 4$. The cohomological dimension of $IA_n$ was recently
computed to be $2n-3$ by Bestvina-Bux-Margalit
\cite{Bestvina-Bux-Margalit_IA}, where the maximal non-zero homology
group was shown not to be finitely generated; this implies in
particular that $IA_3$ is {\it not} finitely presentable.

Even less is known about the groups $J_n^s$ for $s \geq 2$. By
adapting a theorem of McCullough-Miller \cite{McCullough-Miller}, it
follows for $n=3$ and $s \geq 2$ that $J_3^s$ is not finitely
generated, but it is unknown whether or not its abelianization is
finitely generated, and little is known along these lines for $n
\geq 4$. Satoh \cite{Satoh06,Satoh09} has studied some abelian
quotients of terms of the Johnson filtration via the Johnson
homomorphisms (see Section \ref{sec:A review of some central series}
for definitions).

Before stating precise results of this paper, first record some
facts and define some notation. Recall the structure of the free Lie
algebra defined over the integers $\mathbb Z$ generated by a free
abelian group $V$, denoted here by $L[V]$. The free abelian group $V$
is frequently of finite rank.
In the case that the rank is $q$, the notation $V_q$ is used in place of $V$ below.

The Lie algebra $L[V]$ can be described as the smallest sub-Lie
algebra of the tensor algebra $T[V]$ which contains $V$. Thus $L[V]$
is graded with the $s$-th graded summand given by
$$L_s[V]= L[V] \cap V^{\otimes s}.$$ One classical result is that
$L_s[V]$ is a finitely generated free abelian group with ranks given
in \cite{Hall}, \cite{Serre}, and \cite{Milnor-Moore}, with
\cite{Witt} for the graded case.

It is convenient to use the standard dual of $L_s[V]$ given by
$$L_s[V]^* = Hom(L_s[V], \mathbb Z)$$ together with
$$\mathbb Z \oplus L_s[V]^*$$ where $\mathbb Z$ is concentrated
in degree 0, and the finitely generated free abelian group
$L_s[V]^*$ is required to be concentrated in degree $1$ in the main
result stated next.

The $s$-th Johnson homomorphism, as described in Section \ref{sec:A
review of some central series} below,  has domain $J_n^s$ and takes
values in the finitely generated free abelian group $Hom( V_n,
L_{s+1}[V_n])$; thus is a direct product of finitely many copies of the
integers. The integer cohomology of the discrete group $Hom(
V_n, L_{s+1}[V_n])$ is therefore a finitely generated exterior algebra. The
purpose of this article is to show that the image of the Johnson
homomorphism $$\tau_s: J_n^s \to Hom( V_n, L_{s+1}[V_n])$$ on the
level of integer cohomology groups has the direct summand in the image
described next.

\begin{thm}\label{thm:cohomology.in.johnson.filtrations}
If $n \geq 3$, and $2 \leq q \leq  n-1 $, the integral cohomology
ring $H^*(J_n^s)$ contains a direct summand which is additively
isomorphic to $$ \bigotimes_{n-q} (\mathbb Z \oplus L_s[V_{q}]^*)$$
in the image of $$\tau_s^*:H^*(Hom( V_n, L_{s+1}[V_n])) \to
H^*(J_n^s).$$  Thus if $1 \leq i \leq n-2,$ then the ranks of
$H^i(J_n^s)$ increase without bound as $s$ increases.
\end{thm}

\begin{remark}\label{remark:top.cohomology.mccool.plus}
The proof of Theorem \ref{thm:cohomology.in.johnson.filtrations} has direct implications concerning
geometric features of the classifying space of $J_n^s$. Namely, the
suspension of the classifying space $BJ_n^s$ has a wedge summand
given by a bouquet of spheres corresponding to the cohomology
classes in Theorem \ref{thm:cohomology.in.johnson.filtrations}. This
feature translates into a statement about the cohomology of $BJ_n^s$
for any cohomology theory.
\end{remark}

\noindent It follows from the theorem that $$ \bigotimes_{n-q}
L_s[V_{q}]^*$$ is a direct summand of the cohomology group
$H^{n-q}(J_n^s)$, a fact recorded next, in order to exhibit the rapid growth of
the cohomology as the depth of the Johnson filtration $s$ increases.

\begin{cor}\label{cor:top.cohomology.mccool.plus}
For fixed $n > 2 $, and $ n-2 \geq  k \geq 1$, the cohomology group
$H^{k}(J_n^s)$ contains $$ \bigotimes_k({L_s[V_{n-k}]^*})$$ as a
direct summand. Thus specializing to $k = n-2$, the cohomology group
$H^{n-2}(J_n^s)$ contains $\bigotimes_{n-2}(L_s[V_2]^*)$ as a direct
summand.
\end{cor}

Observe that the rank of $L_s[V_q]^*$ grows rapidly for $q > 1$;
estimates of this growth are given at the end Section \ref{sec:The
last step} using Witt's formula for the ranks in free Lie algebras.
Thus the ranks of $H^{i}(J_n^s)$ for $ 1 \leq i \leq n-2$ grow
rapidly for $n$ at least $3$.

The classes constructed in Theorem
\ref{thm:cohomology.in.johnson.filtrations} are fragile in the
following sense, proven in Corollary \ref{cor:unstable.classes}
below, and stated as follows.

\begin{cor}\label{cor:fragile.classes}
If $n \geq 3$, and $2 \leq q \leq  n-1 $, the subalgebra generated
by $\bigotimes_{n-q} L_s[V_{q}]^*$ is in the kernel of the map
$$H^*(J_n^s) \to H^*(J_n^{s+1}).$$
\end{cor}

Note that $Aut(F_n/\Gamma^{s+1})$ acts on the cohomology of $J_n^s$
as this last group is the kernel of the homomorphism $Aut(F_n) \to
Aut(F_n/\Gamma^{s+1})$. However, a more precise description of
algebra generators for  $\bigotimes_{n-q} (\mathbb Z \oplus
L_s[V_{q}]^*)$ gives that this module is not closed with respect to
this action; the closure of this module with respect to the action
is in fact much larger, a feature not addressed here.

The organization of this paper is outlined as follows. Section
\ref{sec:A review of some central series} is a review of the
filtrations of $Aut(F_n)$ given by descending central series and the
Johnson filtration. Subgroups of $Aut(F_n)$ known as McCool's group
are the subject of Section \ref{sec:Properties of McCool's group},
and are used to detect cohomology classes. The descending central
series for McCool's group is the subject of Section \ref{sec:DCS of
McCool}. Section \ref{sec:subgroup} gives natural subgroups of
McCool's group which provide a considerable clarification of the
work here. The values of the Johnson homomorphism on certain
subgroups of McCool's group are recorded in Section \ref{sec:The
values of the Johnson homomorphism on certain subgroups}. Equipped
with a theorem of Falk-Randell \cite[Theorem 3.1]{Falk-Randell}
(appearing here as Theorem \ref{thm:Falk-Randell}), these structures
are used to give a family of cohomology classes for certain
subgroups of McCool's group. These classes are in the image of the Johnson
homomorphism, seen directly by using subgroups of McCool's group
as given in Section \ref{sec:The last step}. That step finishes the
proof of the main theorem.

The authors (weakly) conjecture the following:
\begin{conj*} If $3 \leq n$, $2 \leq s$ and $1 \leq i \leq n-2$, the
cohomology group $H^i(J_n^s)$ is not finitely generated.

\end{conj*}

\section{Central series and the Johnson homomorphisms}\label{sec:A review of some central series}

Recall that the {\it descending central series} of a group $\pi$ is
the  sequence of subgroups
\[ \pi = \Gamma^1(\pi) \supset \Gamma^2(\pi) \supset \cdots \supset  \Gamma^p(\pi)\supset \cdots  \] with $\Gamma^p(\pi) = [
\pi,\Gamma^{p-1}(\pi)].$ It is natural to study the Lie algebra
\[ gr_*^{DCS}(\pi) = \bigoplus_{p \geq 1} gr_p^{DCS}(\pi) \]
associated to its descending central series with the graded terms
given by \[ gr_p^{DCS}(\pi) = \Gamma^p(\pi)/\Gamma^{p+1}(\pi) \]
with bracket given by
\[ [ -, -] : gr_p^{DCS}(\pi) \otimes gr_q^{DCS}(\pi) \mapsto gr_{p+q}^{DCS}(\pi) \]
induced by the commutator on the level of $\pi$ given by
$$[x,y] =
x^{-1}y^{-1}xy, \ x,y \in \pi.$$ In case $\pi$ is  residually
nilpotent, the descending central series filtration is convergent.
Note that the Lie algebra $gr_*^{DCS}(\pi)$ is also graded, but
fails to satisfy the axioms for a graded Lie algebra because of sign
conventions; this failure can be remedied by doubling all degrees to
obtain a graded Lie algebra.

A classical example is that of the free group $F_n$
first investigated by P.~Hall and E.~Witt. Recall that $V =
V_n$ denotes a free abelian group of rank $n$, the first homology
group of $F_n$. As in Section \ref{sec:introduction}, $L[V_n]$
denotes the free Lie algebra generated by $V_n$. Hall \cite{Hall}
and Witt \cite{Witt} proved that
\[ L[V_n] = gr_*^{DCS}(F_n) \] where $V_n$ is a free abelian group of rank $n$ with a choice of basis given by $\{ x_1, \ldots, x_n \}$,
and each $x_i$ is the image under the projection $F_n \to V_n \simeq
H_1(F_n)$ of a basis element for $F_n$.

The graded derivations of a graded Lie algebra inherit the structure
of the graded Lie algebra. In the case of the free Lie algebra
$L[V_n]$, the Lie algebra of graded derivations $$Der(L[V_n])$$ is
additively isomorphic to the direct sum $$\bigoplus_{s \geq 1}
Hom(V_n,L_{s+1}[V_n]).$$ Writing
$$Der_s(L[V_n]) = Hom(V_n,L_{s+1}[V_n]).$$ the Lie bracket is given by
a bilinear pairing
$$Der_s(L[V_n]) \otimes Der_t(L[V_n]) \to Der_{s+t-1}(L[V_n])$$ as
developed by M.~Kontsevich \cite{Kontsevich}; see also
\cite{Conant.Vogtmann} and T.~Jin \cite{Jin}.

\begin{remark}\label{remark:hom.L1} A natural variation is the Lie algebra
$$\widehat{Der}(L[V_n]) = \oplus_{s  \geq 1} Hom(V_n,L_{s}[V_n])$$
given by the direct sum $$Der(L[V_n]) \oplus Hom(V_n,V_n).$$ The
additional group  $$Hom(V_n,V_n) =  End(V_n) = \oplus_{n^2}\mathbb
Z$$ is not used in the computations below.
\end{remark}

Turn now to the Johnson filtration $\{ J_n^s \}$ of $Aut(F_n)$.
Recall that the $s$-th term  $J_n^s$ is the kernel of the``reduction
map'' $$Aut(F_n) \to Aut(F_n/\Gamma^{s+1}).$$ It is well-known that
the intersection of all the terms of the Johnson filtration is
trivial. Furthermore, the successive quotients $gr_s^J(IA_n) =
J_n^s/J_n^{s+1}$ are torsion-free finitely generated abelian groups.
The quotients $gr_s^J(IA_n)$ give a graded Lie algebra which is
free as a $\BZ$-module, with bracket inherited from the commutator.
The direct sum of these quotients
\[ gr_*^J(IA_n) = \bigoplus_{s \geq 1} gr_s^J(IA_n) \]
admits a natural structure of a Lie algebra, called the {\it Johnson
Lie algebra} of the group $IA_n$ \cite{Andreadakis,
Cohen-Pakianathan, Farb, Kawazumi}.

The first {\it Johnson homomorphism} $\tau_1$ on $IA_n$ is defined
by
\begin{equation*}
\tau_1: IA_n \longrightarrow Hom(V_n,L_2[V_n])
\end{equation*}
by (with some abuse of notation) $\tau_1(\phi)(w)$ equal to the image of
$\phi(\tilde{x})\tilde{x}^{-1}$ in $L_2[V_n]$, where $\phi \in IA_n$, the element $x \in
V_n$, and $\tilde{x}$ is a lift of $x$ to $F_n/\Gamma^3 F_n$. It is
straightforward to check that the kernel of $\tau_1$ is $J_n^2$.
Define inductively the $s$-th Johnson homomorphism on the kernel
$J_n^s$ of $\tau_{s-1}$ by
\begin{equation*}\label{equation*:Johnson.homomorphism}
\tau_s: J_n^s \longrightarrow Hom(V_n,L_{s+1}[V_n])
\end{equation*}
with $\tau_s(\phi)(x) = \phi(\tilde{x})\tilde{x}^{-1}$ for $\phi \in
J_n^s$, and $x \in H_1(F_n)$ for any lift $\tilde{x} \in
F_n/\Gamma^{s+2}F_n$. As the group $J_n^{s+1}$ is precisely the kernel
of
$$\tau_s: J_n^s \to Hom(V_n,L_{s+1}[V_n])$$
there is an induced map $$\tau_s: gr_s^J(IA_n)= J_n^s/J_n^{s+1} \to
Hom(V_n,L_{s+1}[V_n]).$$ Passing to direct sums, there is an induced
map $$\bigoplus_{s \geq 1}\tau_s: \bigoplus_{s \geq 1} J_n^s \to
\bigoplus_{s \geq 1} Hom(V_n,L_{s+1}[V_n]).$$ With the
identification of the Lie algebra of graded derivations
$Der(L[V_n])$ with $\oplus_{s \geq 1} Hom(V_n,L_{s+1}[V_n])$, the
induced homomorphism $$J: gr_*^J(IA_n) \to Der(L[V_n])$$ is a
morphism of Lie algebras \cite{Andreadakis, Cohen-Pakianathan, Farb,
Kawazumi}.

Thus, there are two natural structures of Lie algebras for $IA_n$
given by $gr_*^{DCS}(\IA)$ and $gr_*^J(\IA_n)$.

\section{Properties of McCool's group}\label{sec:Properties of McCool's group}

Recall Magnus's generating set for $IA_n$ \cite{Magnus}, consisting
of automorphisms
\[ \mathcal{M}_n = \{ \alpha_{ij} \ | \ i \neq j \} \cup \{ A_{ijk} \ | \ i \neq j,k; \ j<k \} \]
where
\begin{eqnarray*}
\alpha_{ij}(x_r) = \left \{ \begin{array}{ll}
x_r & r \neq i \\
x_j x_r x_j^{-1} & r=i
\end{array}
\right.
\end{eqnarray*}

\begin{eqnarray*}
A_{ijk}(x_r) = \left \{ \begin{array}{ll}
x_r & r \neq i \\
\lbrack x_j, x_k \rbrack x_i & r=i.
\end{array}
\right.
\end{eqnarray*}

McCool \cite{McCool} proved that the subgroup $\PS_n$ of {\it pure
symmetric automorphisms} (or the {\it McCool group}) of $IA_n$,
consisting of those automorphisms which map each generator $x_i$ to
a conjugate of itself, is generated by the subset of Magnus
generators
\begin{equation*}
\PS_n = \langle \alpha_{ij} \ | \ i \neq j \rangle.
\end{equation*}
McCool provided a finite presentation of $\PS_n$ in terms of these
generators. The group $\PS_n$ is interesting to topologists as it
appears as the mapping class group of the complement of $n$ unlinked
circles in $\BR^3$, and is thus a  generalization of the pure braid
group (see \cite{Goldsmith}, for example). The pure braid group is
itself realized as a subgroup of $\PS_n$. McCool proved the
following theorem.

\begin{thm} \label{thm:McCool relations}
A presentation of $P\Sigma_n$ is given by generators $\alpha_{k,j}$
together with the following relations.
\begin{enumerate}
\item $\alpha_{i,j}\cdot \alpha_{k,j}\cdot \alpha_{i,k} = \alpha_{i,k} \cdot
\alpha_{i,j} \cdot \alpha_{k,j}$ for $i,j,k$ distinct.

\item $[\alpha_{k,j}, \alpha_{s,t}] = 1$ if $\{j,k\} \cap \{s,t \}=
\phi$.

\item $[\alpha_{i,j},\alpha_{k,j}] = 1$ for
$i,j,k$ distinct.

\item $[\alpha_{i,j} \cdot \alpha_{k,j}, \alpha_{i,k}] = 1$ for
$i,j,k$ distinct (redundantly).

\end{enumerate}
\end{thm}

Consider the subgroup $\PS_n^+$ of $\PS_n$ generated by the subset
\begin{eqnarray*}
\PS_n^+ = \langle \alpha_{ij} \ | \ 1 \leq j < i \leq n \rangle.
\end{eqnarray*}
This group is referred to as the {\it upper triangular McCool group}
in \cite{cvwp}. There exists an exact sequence
\begin{equation*}
1 \longrightarrow F_{n-1} \longrightarrow \PS_n^+ \longrightarrow
\PS_{n-1}^+ \longrightarrow 1
\end{equation*}
where the induced action of $\PS_{n-1}^+$ on $H_1(F_{n-1})$ is
trivial. This fact is recorded in \cite{cvwp} as Lemma
\ref{lem:sequence-action}, and it should be compared with the case
of the pure braid group (See \cite{frc}, pages 281--283, or
\cite{Falk-Randell} for instance.).

The map $\pi:\PS_n^+ \to \PS_{n-1}^+$ in \cite{cvwp} is recalled
next for the convenience of the reader, as this map is used below.
\[
\pi(\alpha_{k,i}) =
\begin{cases}
\alpha_{k,i} & \text{if $i < n$, and $k < n$ ,}\\
1 & \text{if $i = n$ or $k=n$.}
\end{cases}
\] Furthermore, there is a natural cross-section for $\pi$, $$\sigma: \PS_{n-1}^+ \to \PS_{n}^+ $$ defined by
$$\sigma(\alpha_{k,i}) = \alpha_{k,i}.$$ Thus the group $\PS_{n-1}^+$ is regarded as a subgroup of
$\PS_{n}^+$ below.

\section{The descending central series}\label{sec:DCS of McCool}

The purpose of this section is to develop properties of the
descending central series for $\PS_n^+$ by using the following
theorem of Falk-Randell \cite{Falk-Randell}:
\begin{thm}[Falk-Randell \cite{Falk-Randell}, Theorem 3.1]
\label{thm:Falk-Randell} Suppose that $1 \to A \to B \to C \to 1$ is
a split exact sequence of groups, and the induced conjugation action
of $C$ on $H_1(A)$ is trivial (that is, $[A,C] \subset [A,A]$). Then
the sequence of induced maps
\begin{equation*}
1 \longrightarrow \Gamma^s A \longrightarrow \Gamma^s B
\longrightarrow \Gamma^s C \longrightarrow 1
\end{equation*}
is split exact for every $s$.

Furthermore, there is an induced exact sequence on the level of
associated graded modules
\begin{equation*}
0 \longrightarrow gr_s^{DCS}(A) \longrightarrow gr_s^{DCS}(B)
\longrightarrow gr_s^{DCS}(C) \longrightarrow 0
\end{equation*} which is additively split.

\end{thm}

To apply Theorem \ref{thm:Falk-Randell} to the group $\PS_n^+$,
consider a lemma given by the second conclusion of Theorem $1.2$ of
\cite{cvwp}.
\begin{lem}\label{lem:sequence-action}
The sequence
\begin{equation*}
1 \longrightarrow F_{n-1} \longrightarrow \PS_n^+ \longrightarrow
\PS_{n-1}^+ \longrightarrow 1
\end{equation*}
induced by the map $x_n \mapsto {\rm id}_{F_{n-1}}$ is split exact,
where $F_{n-1}$ is the free group on the generators
\begin{equation*}
\alpha_{n,1}, \ldots, \alpha_{n,n-1}
\end{equation*}
Furthermore the action of $\PS_{n-1}^+$ on $H_1(F_{n-1})$ is
trivial.
\end{lem}

A consequence of Lemma \ref{lem:sequence-action} follows.
\begin{lem}\label{lem:exact sequence}
The sequence
\begin{equation}\label{exact sequence}
1 \lra \Gamma^s F_{n-1} \lra \Gamma^s \PS_n^+ \lra \Gamma^s
\PS_{n-1}^+ \lra 1
\end{equation}
is split exact.
\end{lem}
\noindent The next proposition then follows directly.
\begin{prop}\label{exact-sequences}
If $s \geq 2$, then there is an additive isomorphism
\begin{equation*} \bigoplus_{q=2}^n L_s[V_{q-1}] \to \Gamma^s \PS_n^+/\Gamma^{s+1} \PS_n^+.
 \end{equation*}
Furthermore, each free Lie algebra $L[V_{q-1}]$ is a sub-Lie algebra
of $gr_*(\PS_n^+)$, and there is an induced isomorphism of abelian
groups $$\bigoplus_{q=2}^n L[V_{q-1}] \to gr_*(\PS_n^+)$$ (where
this isomorphism does not preserve the structure as Lie algebras).

\

If $s =1$, then there is an additive isomorphism
\begin{equation*} \bigoplus_{q=2}^n V_{q-1} \to \Gamma^1 \PS_n^+/\Gamma^{2} \PS_n^+ = H_1(\PS_n^+).
 \end{equation*}
\end{prop}

\section{On a subgroup}\label{sec:subgroup}

The purpose of this section is to define subgroups of $\PS_n^+$,
denoted $H(n,k)$ and $G(n,k,j)$; these groups will appear in
the proof of the main Theorem.

\begin{definition}\label{def:subgroups}
Fix integers $ 1 \leq j \leq k-1 \leq n-1$, and define
$$G(n,k,j)$$ as the subgroup of $\PS_n^+$ generated by the elements
$$\alpha_{k,1},\alpha_{k,2}, \cdots, \alpha_{k,j}.$$

The group $$H(n,k)$$ is defined to be the direct product
$$G(n,k,k-1)\times G(n,k+1,k-1)\times \cdots \times G(n,n,k-1).$$
\end{definition}

Properties of the groups $G(n,k+r,k-1)$, and $H(n,k)$ are recorded
next.
\begin{lem}\label{lem:commuting.elements}
The groups $G(n,k+r,k-1)$ are free subgroups of the group $\PS_n^+$
for $2 \leq k \leq k+r \leq n$. Furthermore, if
\begin{enumerate}
      \item $x \in G(n,k+s,k-1)$ for $0 \leq s \leq n-k$, and
      \item $y \in G(n,k+r,k-1)$ for $ 0 \leq s  < r \leq n - r$, then
    \end{enumerate} $$xy=yx,$$ and there are induced homomorphisms
$$\Theta(n,k): G(n,k,k-1)\times G(n,k+1,k-1)\times \cdots \times G(n,n,k-1) \to
\PS_n^+$$ such that $$\Theta(n,k)(\alpha_{t,s}) = \alpha_{t,s}\in
G(n,k+r,k-1)$$ for
$$ 0 \leq r \leq n-k.$$

\end{lem}

\begin{proof} Recall from \cite{cvwp} that there are split
epimorphisms $$\pi:\PS_n^+ \to \PS_{n-1}^+$$ with kernel given by a
free group on $(n-1)$-letters with basis
$$\alpha_{n,1},\alpha_{n,2}, \cdots, \alpha_{n,n-2}, \alpha_{n,n-1}.$$
Thus $G(n,n,n-1)$ is a free subgroup of $\PS_n^+$ for all $2 \leq
n$, and the first assertion that each $G(n,k+r,k-1)$ is free
follows.

Next, assume that
\begin{enumerate}
      \item $x \in G(n,k+s,k-1)$ for $0 \leq s \leq n-k$, and
      \item $y \in G(n,k+r,k-1)$ for $ 0 \leq s  < r \leq n - r$.
    \end{enumerate}

It suffices to check the next assertion that $xy = yx$ in case
$$x = \alpha_{k+s,u}, \ u \leq k-1,$$ and
$$y = \alpha_{k+r,t}, \ s \leq r \leq n-k, \ t \leq k-1.$$

Observe that either
\begin{enumerate}
  \item $s = t$ so that $[\alpha_{k+r,s},\alpha_{k,s}] = 1$ by Theorem \ref{thm:McCool relations}, or
  \item  $s \neq t$ in which case $[\alpha_{k+r,t},\alpha_{k,s}] =
  1$ by Theorem \ref{thm:McCool relations}
  as $s,t \leq k-1$ while the sets $\{k,s\}$ and $\{k+r,t\}$ are
  disjoint.
\end{enumerate}
\end{proof}

The next Lemma follows by inspection of the definitions.
\begin{lem}\label{lem:inclusions}
If  $2 \leq k \leq k+r \leq n-1$, then the groups $G(n,k+r,k-1)$ are
subgroups of the group $\PS_{n-1}^+ \subset \PS_{n}^+$, and
$$G(n,k+r,k-1) = G(n-1,k+r,k-1).$$
\end{lem}

Next, consider the projection maps $$p: H(n,k) \to G(n,k,k-1)\times
G(n,k+1,k-1)\times \cdots \times G(n,n-1,k-1)$$ which deletes the
coordinate in $G(n,n,k-1)$. Observe that $G(n,k+1,k-1)\times \cdots
\times G(n,n-1,k-1)$ is equal to $H(n-1,k-1) \subset \PS_{n-1}^+.$
Furthermore, the maps
$$\Theta(n,k): G(n,k,k-1)\times G(n,k+1,k-1)\times \cdots \times
G(n,n,k-1) \to \PS_n^+$$ are compatible with the projection maps of
$$\pi:\PS_n^+ \to \PS_{n-1}^+$$ of \cite{cvwp} in the following
sense.

\begin{lem}\label{lem:Projection.maps}

If $ 2 \leq k <n$, the group $\Gamma^s(H(n,k))$ is isomorphic to the
direct product $$\Gamma^s(G(n,k,k-1))\times
\Gamma^s(G(n,k+1,k-1))\times \cdots \times \Gamma^s(G(n,n,k-1)).$$

There is a morphism of group extensions
\[
\begin{CD}
G(n,n,k-1) @>{\Theta(n,k)}>> F_{n-1}     \\
  @V{i}VV                    @VV{i}V \\
H(n,k) @>{\Theta(n,k)}>> P\Sigma_{n}^+    \\
  @V{p}VV                    @VV{\pi}V \\
H(n-1,k) @>{\Theta(n-1,k-1)}>> P\Sigma_{n-1}^+ .
\end{CD}
\]

Furthermore, there are induced maps of the level of the $s$-th stage
of the descending central series
\[
\begin{CD}
\Gamma^s(G(n,n,k-1)) @>{\Theta(n,k)}>> \Gamma^s(F_{n-1})     \\
  @V{i}VV                    @VV{i}V \\
\Gamma^s(H(n,k)) @>{\Theta(n,k)}>> \Gamma^s(P\Sigma_{n}^+)   \\
  @V{p}VV                    @VV{\pi}V \\
\Gamma^s(H(n-1,k)) @>{\Theta(n-1,k-1)}>> \Gamma^s(P\Sigma_{n-1}^+),
\end{CD}
\] and the vertical columns are group extensions.

\end{lem}

\begin{proof} The first assertion concerning the product decomposition of
$\Gamma^s(H(n,k))$ follows from the fact that $H(n,k)$ is a product. That the first diagram commutes follows from the definition of the
map $p:H(n,k) \to H(n-1,k)$.

The third assertion concerning the group extensions as well as
stages of the descending central series follows by naturality for the $H(n,k)$ and by the Falk-Randell theorem, stated here as Proposition \ref{lem:exact
sequence}.
\end{proof}

Since $H(n,k)$ is a direct product of $(n-k)$ free groups, each of
which have $(k-1)$ generators, the next corollary follows at once.

\begin{cor}\label{cor:Lie.algebra.for.H}
If $ 2 \leq k \leq n$, the Lie algebra $gr_*^{DCS}H(n,k)$ is
isomorphic to the direct sum of Lie algebras $$\bigoplus_{k \leq m
\leq n} gr_*^{DCS}G(n,m,k-1) \cong \bigoplus_{n-k}L[V_{k-1}]$$ with
generators for the $m$-th summand represented by
$$\alpha_{m,1},\alpha_{m,2}, \cdots, \alpha_{m,k-1}$$ for all $n \geq m
\geq k$.
\end{cor}

\begin{lem}\label{lem:split monic}
If $ 2 \leq k \leq n$, the map
\[
\begin{CD}
\Gamma^s(H(n,k)) @>{\Theta(n,k)}>> \Gamma^s(P\Sigma_{n}^+)
\end{CD}
\] of Lemma \ref{lem:Projection.maps} induces a monomorphism of Lie algebras

\[
\begin{CD}
gr_*^{DCS}(H(n,k)) @>{gr_*(\Theta(n,k))}>> gr_*^{DCS}(P\Sigma_{n}^+)
\end{CD}
\] which is a split monomorphism of abelian groups.

\end{lem}

\begin{proof} Observe that $gr_s^{DCS}(P\Sigma_{n}^+)$ was computed in
Proposition \ref{exact-sequences}.  Corollary \ref{cor:Lie.algebra.for.H} states that
if $ 2 \leq k \leq n$, the Lie algebra $gr_*^{DCS}H(n,k)$ is
isomorphic to the direct sum of Lie algebras $\bigoplus_{k \leq m \leq n} \ gr_*^{DCS}G(n,m,k-1) \cong \bigoplus_{n-k}L[V_{k-1}]$
with generators for the $m$-th summand
represented by
$$\alpha_{m,1},\alpha_{m,2}, \cdots, \alpha_{m,k-1}$$ for all $m
\geq k$.

Since the map $\Theta(n,k):H(n,k) \to \PS_n^+$ restricts to a map
$$\Theta(n,k):G(n,k+r,k-1) \to \PS_{k+r}^+$$ for every $ 0 \leq r
\leq n-k$, it suffices to check that the induced map $G(n,n,k-1) \to
F_{n-1}$, where $F_{n-1}$ is the kernel of $\pi:\PS_n^+ \to
\PS_{n-1}^+$, induces a split monomorphism on the level of Lie
algebras.

Note that $G(n,n,k-1)$ is the subgroup of $\PS_n^+$ generated by the
elements $$\alpha_{n,1},\alpha_{n,2}, \cdots, \alpha_{n,k-1}.$$ Thus
the inclusion $G(n,n,k-1) \to F_{n-1}$ is a split monomorphism on
the level of free groups, and hence on the level of Lie algebras.
Thus, the induced map $$gr_s^{DCS}H(n,k)\to
gr_s^{DCS}(P\Sigma_{n}^+)$$ is a split monomorphism as it is a
direct sum of maps which are monomorphisms of Lie algebras each of
which is split as abelian groups.
\end{proof}

The statement and proof of the next standard fact are recorded for
the convenience of the reader.
\begin{lem}\label{lem:observation}
Let $\Gamma^s(F_q)$ denote the $s$-th stage of the descending
central series for the free group $F_q$. Then $$\mathbb Z \oplus
L_s[V_q]^*$$ is a direct summand of the cohomology of
$\Gamma^s(F_q)$.

\end{lem}

\begin{proof}
Recall that the Lie algebra attached to the descending central
series of $F_q$ is the free Lie algebra $L[V_q]$ with the $s$-th
graded direct summand given by
$$L_s[V_q] = \Gamma^s(F_q)/
\Gamma^{s+1}(F_q).$$ Thus there is a group extension
\[
\begin{CD}
1  @>{}>> \Gamma^{s+1}(F_q)@>{}>> \Gamma^{s}(F_q)@>{}>> L_s[V_q] @>{}>> 1. \\
\end{CD}
\]

Since $L_s[V_q]$ is a finitely generated free abelian group, it has
a basis over the integers. That basis depends on both $s$ and $q$.
Thus let $\mathbb S(s,q)$ denote a set which indexes this basis. Fix
a choice of basis $b_{\alpha}$, $\alpha \in \mathbb S(s,q)$, and let
$F$ denote a free group with this choice of basis.

There is a choice of lift of $b_{\alpha}$ to $\gamma_{\alpha} \in
\Gamma^{s}(F_q)$ for each $\alpha \in \mathbb S(s,q)$. Thus there is
an induced homomorphism $$\Theta: F \to \Gamma^{s}(F_q)$$ with the
property that the composite
\[
\begin{CD}
F @>{\Theta}>> \Gamma^{s}(F_q)@>{}>> L_s[V_q] \\
\end{CD}
\] is an epimorphism.

\

Thus this composite map induces an isomorphism
\[
\begin{CD}
H_1(F)  @>{}>> H_1(L_s[V_q]) =
L_s[V_q] \\
\end{CD}
\] as well as an isomorphism

\[
\begin{CD}
H^1(L_s[V_q]) = L_s[V_q]^* @>{}>> H^1(F). \\
\end{CD}
\]

It follows that $H^1(L_s[V_q]) = L_s[V]^*$ injects in
$H^1(\Gamma^{s}(F_q))$ by inspection. Furthermore, this injection is
split by the map induced in cohomology from $\Theta: F \to
\Gamma^{s}(F_q)$. Thus this gives a direct summand.
\end{proof}

\begin{remark}\label{rem:homology.of.quotients}
\

\begin{enumerate}
  \item The cohomology of $F_q/\Gamma^s(F_q)$ is not yet well understood
  for large $q$. One classical result of Hopf's theorem about the
  second homology group of a discrete group is an isomorphism
  $$H_2(F_q/\Gamma^s(F_q)) \cong L_{s+1}[V_q].$$
  \item One feature concerning Lemma \ref{lem:observation} is developed next. The group
$\Gamma^s(F_q)$ is a free group. However, control of the generators
or even the first homology group is tenuous (as seen from the
formulae of Witt for the ranks). Indeed, the first homology group is
generally much larger than the group $L_s[V]$. However, this last
group $L_s[V]$  is a direct summand of $H_1(\Gamma^s(F_q))$, which
allows easy manipulation in this context and forces the rapid growth
of the cohomology of $J_n^s$.
\end{enumerate}

\end{remark}

The utility of Lemma \ref{lem:observation} is as follows.

\begin{cor}\label{cor:cohomology.of.Lie.algebra.for.H}
Assume that $ 2 \leq k \leq n$.

\begin{enumerate}
  \item The integral cohomology algebra of $H(n,k)$ is isomorphic to
 $$ \bigotimes_{n-k} (\mathbb Z \oplus L_1[V_{k-1}]^*) = \bigotimes_{n-k} (\mathbb Z \oplus V_{k-1}^*).$$
  \item  The integral cohomology algebra of $\Gamma^sH(n,k)$
  contains a subalgebra which is isomorphic
 $$ \bigotimes_{n-k} (\mathbb Z \oplus L_s[V_{k-1}]^*),$$ which is the image of the natural map
 $$H^*(gr_s^{DCS}H(n,k)) \to H^*(\Gamma^sH(n,k)).$$

\end{enumerate}
\end{cor}

\section{Values of the Johnson homomorphism on certain subgroups}\label{sec:The values of the Johnson homomorphism on certain subgroups}

The purpose of this section is to derive the values of the Johnson
homomorphism for certain subgroups of $\PS_n^+$. Recall from Section
\ref{sec:Properties of McCool's group} that $\PS_n^+$ is generated
by the subset of Magnus generators
$$ \{ \alpha_{ij} \ | \ 1 \leq j < i \leq n \}.$$ For ease of
notation, fix $q$ and write $$w_r= \alpha_{qr}$$ for $q>r$. A fact
observed in \cite{cvwp} is that the elements $w_r= \alpha_{qr}$ for
$1 \leq r \leq q-1$ give a basis for a free group in $\PS_n^+$.

\

Throughout the remainder of this section, it will be tacitly assumed
that $$1 \leq r \leq q-1.$$ The formulae are verified as follows.
First consider the action of $$w_r= \alpha_{qr}$$ for $q>r$ on
$x_t$:
\[
w_r(x_t) =
\begin{cases}
x_t & \text{if $t \neq q$,}\\
x_rx_qx_r^{-1}
 & \text{if $t=q$.}
\end{cases}
\]

Thus for example $$1 \leq r_i \leq q-1,$$

\[
(w_{r_1}w_{r_2})(x_t) = w_{r_1}((w_{r_2}(x_t)) =
\begin{cases}
w_{r_1}(x_t) = x_t & \text{if $t \neq q$,}\\
w_{r_1}( x_{r_2}x_qx_{r_2}^{-1}) = x_{r_1}
(x_{r_2}x_qx_{r_2}^{-1})x_{r_1}^{-1}
 & \text{if $t=q$.}
\end{cases}
\]

Next, consider a product given by
$$\mathcal W = w_{r_1}^{\epsilon_1}\cdot
w_{r_2}^{\epsilon_2}\cdots w_{r_m}^{\epsilon_m}$$ for $\epsilon_i =
\pm 1$ for $1 \leq r_i \leq q-1.$ We begin to record some formulae
essential to our computations in the sequel.

\begin{lem}\label{lem:action.of.W}

If $$\mathcal W = w_{r_1}^{\epsilon_1}\cdot
w_{r_2}^{\epsilon_2}\cdots w_{r_m}^{\epsilon_m}$$ for $\epsilon_i =
\pm 1$ for $1 \leq r_i \leq q-1$, the action of $\mathcal W$ is
specified by the formula

\[
\mathcal W(x_t) =
\begin{cases}
x_t & \text{if $t \neq q$,}\\
(x_{r_1}^{\epsilon_1}\cdot x_{r_2}^{\epsilon_2}\cdots
x_{r_m}^{\epsilon_m})\cdot x_q \cdot (x_{r_1}^{\epsilon_1}\cdot
x_{r_2}^{\epsilon_2}\cdots x_{r_m}^{\epsilon_m})^{-1} & \text{if
$t=q$.}
\end{cases}
\]

\

\end{lem}

Thus for example, the action of the commutator $$[w_{r_1},w_{r_2}] =
w_{r_1}^{-1}w_{r_2}^{-1}w_{r_1}w_{r_2}$$ on $x_t$ is specified by

\[
[w_{r_1},w_{r_2}](x_t) =
\begin{cases}
x_t & \text{if $t \neq q$,}\\
[x_{r_1},x_{r_2}]\cdot x_q \cdot [x_{r_1},x_{r_2}]^{-1} & \text{if
$t=q$.}
\end{cases}
\]
The formula for the action of the commutator $$\Lambda = [\cdots
[w_{r_1}, w_{r_2}] \cdots ] w_{r_{m}}] \in IA_n$$ on $x_t$ for $1
\leq r_i \leq q-1$ is thus given by the formula

\[
\Lambda(x_t) =
\begin{cases}
x_t & \text{if $t \neq q$,}\\
\Lambda_x \cdot x_q \cdot \Lambda_x^{-1} & \text{if $t=q$}
\end{cases}
\] where $$\Lambda_x = [\cdots [x_{r_1}, x_{r_2}] \cdots ] x_{r_{m}}],$$ the commutator formally obtained by replacing
each $w_{r_i}$ by $x_{r_i}$ in the commutator $\Lambda$.

\

Values resulting from applying the Johnson homomorphism are recorded
next.
\begin{prop}\label{thm:johnson.on.upper.triangular.mccool}
Consider the commutator $\Lambda = [\cdots [w_{r_1}, w_{r_2}],
\cdots ], w_{r_{t}}]$.  If  $ r_1, r_2, \ldots, r_{t} < q$, then

\[
\tau_s(\Lambda)(x_t) =
\begin{cases}
x_t & \text{if $t \neq q$,}\\
\Lambda_x \cdot x_q \cdot \Lambda_x^{-1}\cdot x_{q}^{-1} =
[\Lambda_x^{-1},{x_q}^{-1}]
 & \text{if $t=q$.}
\end{cases}
\] \qed
\end{prop}

The next statement records implications of these formulae on the
level of Lie algebras.
\begin{cor}\label{inclusion}
The composite morphism of Lie algebras denoted

$$J:\bigoplus_{s \geq 1} gr_s^{DCS}(\PS_n^+) \to Der(L[V_n])$$ given
by
\begin{equation*}
\bigoplus_{s \geq 1} gr_s^{DCS}(\PS_n^+) \longrightarrow
\bigoplus_{s \geq 1} gr_s^J(IA_n) \longrightarrow Der(L[V_n])
\end{equation*} which is induced by the Johnson homomorphisms is injective, and is split
injective as abelian groups (but not split as Lie algebras).
\end{cor}

\section{The last step}\label{sec:The last step}

By Corollary \ref{inclusion}, the composite morphism of Lie algebras
$$J: gr_*^{DCS}(\PS_n^+) \to Der(L[V_n])$$ given
by
\begin{equation*}
\bigoplus_{s \geq 1} gr_s^{DCS}(\PS_n^+) \longrightarrow
\bigoplus_{s \geq 1} gr_s^J(IA_n) \longrightarrow Der(L[V_n])
\end{equation*} is injective, and is additively split. By \ref{lem:split monic}, the morphism of Lie algebras

\[
\begin{CD}
gr_*^{DCS}(H(n,k)) @>{gr_*(\Theta(n,k))}>>
gr_*^{DCS}(P\Sigma_{n}^+).
\end{CD}
\] is a monomorphism which is additively split
in case $ 2 \leq k \leq n$. The next theorem follows at once.

\begin{thm}\label{thm:more.morphisms.of.Lie.algebras}
If $n \geq 3$, and $ 2 \leq k \leq n$, the composite homomorphism
$$ H(n,k) \to  P\Sigma_{n}^+ \to IA_n$$ induces a morphism of Lie algebras
$$gr_*^{DCS}(H(n,k)) \to gr_*^{DCS}(P\Sigma_{n}^+)\to  gr_*^J(IA_n)\to Der(L[V_n]).$$
This composite is a monomorphism of Lie algebras and  is a split
monomorphism of abelian groups.
\end{thm}

Since the composite map of Theorem
\ref{thm:more.morphisms.of.Lie.algebras}
$$\gamma:gr_s^{DCS}(H(n,k)) \to Hom(V_n, L_{s+1}[V_n])$$ is a split monomorphism of finitely generated, free abelian groups,
the map $\gamma$ induces a split epimorphism in integer cohomology
$$\gamma^*:H^*(Hom(V_n, L_{s+1}[V_n])) \to H^*(gr_s^{DCS}(H(n,k))).$$

Observe that the cohomology ring of $gr_s^{DCS}(H(n,k))$ is
isomorphic to that of the product $$\Gamma^s(G(n,k,k-1))\times
\Gamma^s(G(n,k+1,k-1))\times \cdots \times \Gamma^s(G(n,n,k-1))$$ by
Lemma \ref{lem:Projection.maps}. Thus, the cohomology of
$\Gamma^sH(n,k)$ contains
$$ \bigotimes_{n-k+1} (\mathbb Z \oplus L_s[V_{k-1}]^*)$$ by
Lemma \ref{lem:observation}. On the other-hand, the natural quotient
map
\[
\begin{CD}
\Gamma^sH(n,k) @>{gr_s(\Theta(n,k))}>> gr_s^{DCS}(H(n,k))
\end{CD}
\] induces a surjection onto its image in cohomology given in
Corollary \ref{cor:cohomology.of.Lie.algebra.for.H} by $$
\bigotimes_{n-k+1} (\mathbb Z \oplus L_s[V_{k-1}]^*).$$

The next statement as well as the main Theorem
\ref{thm:cohomology.in.johnson.filtrations} follows by setting $q =
k-1$: the case $k=2$ is deleted as $L_s[V_{1}]^* = \{0\}$ for $s>1$.
\begin{thm}\label{thm:cohomology.estimate}
If $n \geq 3$, and $ 3 \leq k \leq n$, the integral cohomology ring
$H^*(J_n^s)$ contains a direct summand which is additively
isomorphic to
$$\bigotimes_{n-k+1} (\mathbb Z \oplus L_s[V_{k-1}]^*).$$

Furthermore, this summand is in the image of the map induced by the
Johnson homomorphism on integral cohomology groups
$$(\tau_s)^*:H^*(Hom( V_n, L_{s+1}[V_n])) \to H^*(J^s_n).$$
\end{thm}

Since the composite

\[
\begin{CD}
J^{s+1}_n @>{}>>J^s_n @>{\tau_s}>>  Hom( V_n, L_{s+1}[V_n])
\end{CD}
\] is constant by definition of the Johnson homomorphism (as given in Section
\ref{sec:A review of some central series}), the next result, the
`fragility' of these cohomology classes, follows at once.
\begin{cor}\label{cor:unstable.classes}
If $n \geq 3$, and $ 2 \leq k \leq n$, the composite
\[
\begin{CD}
J^{s+1}_n @>{}>>J^s_n @>{\tau_s}>>  Hom( V_n, L_{s+1}[V_n])
\end{CD}
\] gives the trivial map in cohomology when restricted to
$$\bigotimes_{n-k+1} L_s[V_{k-1}]^*,$$ and these classes are
in the kernel of the map  $$H^*(J_n^s) \to H^*(J_n^{s+1}).$$
\end{cor}

To estimate the ranks of the free abelian groups $L_s[V_{q}]^*$ for
fixed $1 \leq q \leq n-1$ where $V_q = \oplus_q \mathbb Z =
H_1(F_q),$ classical work of Witt is recalled next
\cite{Serre,Witt}. For fixed filtration degree $s$, write $d_s(V)$
for the rank of the free abelian group  $L_s[V]$ occurring above and
consider the power series $$\sum_{s \geq 0}  d_s(V)t^s$$ where by
convention $$d_0(V) = 1.$$

\noindent Thus $$d_1(V) = q.$$

Following Witt's application of the Poincar\'e-Birkhoff-Witt
theorem, $$1/(1-qt) = \prod_{s \geq 1} 1/{(1-t^s)^{d_s(V)}}.$$ To
find an inductive formula for the coefficients $d_s(V)$, take formal
logarithms of both sides of this equation to obtain the formula
$$q^s = \sum_{m|s} m d_m(V).$$ An elegant exposition for this information
is in Serre's book \cite{Serre}. Observe that $$sd_s(V)= q^s -
\sum_{m|s, \ m < s} m d_m(V).$$

Next, specialize to the case of filtration degree $s$ for the
summand $L_s[V]$ with $$s = p^r, \ p \ \hbox{is assumed to be
prime.}$$ The formula $q^s = \sum_{m|p^r} m d_m(V)$ then simplifies
to
$$q^{p^r} = \sum_{0 \leq i \leq r} p^{i} d_{p^i}(V) =
d_1(V) +pd_p(V) + \cdots + p^{r-1}d_{p^{r-1}}(V) + p^rd_{p^r}(V).$$

To illustrate this computation, some values are
listed next.
\begin{eqnarray*}
d_{p^{r}}(V) = \left \{ \begin{array}{ll}
q & r =0, \\
(q^{p^{r}} - q^{p^{r-1}})/p^r, & r > 0.
\end{array}
\right.
\end{eqnarray*} Thus in case $q > p$ for a fixed prime
$p$, the previous formula illustrates the rapid growth of the values
$d_{p^{r}}(V)$.

\section{Further comparison with earlier work} \label{sec:further.comparison with earlier work}

This section consists of a remark concerning work of M.~Bestvina,
K.~Bux, and D.~Margalit \cite{Bestvina-Bux-Margalit_IA}. They
exhibit an abelian subgroup of $IA_n$ determined by the
automorphisms $$\rho(p_j,q_j): H_n \to F_n$$ defined by

$$x_1 \to x_1$$
$$x_2 \to x_2$$
$$x_j \to w^{p_j}x_jw^{q_j}$$ where $w= [x_1,x_2], \ j > 2.$

Depending on the choices of $p_j$ and $q_j$, these elements live in
various stages of the Johnson filtrations. For example, if $$p_j =
1, \ \hbox{and} \ q_j = -1,$$ then the elements $\rho(p_j,q_j)$ are
in $\Gamma^2 \PS_n^+$. The groups $H(n,k)$ of section
\ref{sec:subgroup} give non-trivial abelian subgroups in $J_n^s$ for
large $s$. It is natural to ask whether the above methods imply that
$H^i(J_n^s)$ fails to be finitely generated as long as $n > 2$, $s >
2$, and $2 \leq i \leq n-2$.

\section{Appendix} \label{sec:Appendix}

The purpose of this section is to list the natural
Euler-Poincar\'e series associated to the Lie algebra
$$Der(L[V_n]) = \bigoplus_{ 1 \leq s} Hom(V_n,L_{s+1}[V_n])$$ where
each module $Hom(V_n,L_{s+1}[V_n])$ is formally assigned gradation
$s$.

The reason for doing so is that these modules are the natural images
of the Johnson homomorphism, which is injective. Thus these maps are
split rationally, and so the computation given next may provide a
setting for enumerating the cokernel of the Johnson homomorphisms in
a `global' way.

\

Recall that the rank of $V_n$ is $n$, and the rank of $L_{s}[V_n]$
is $$d_s(V_n)$$ subject to the relations discovered by Witt as
described in Section \ref{sec:The last step}. Thus the natural
Euler-Poincar\'e series associated to the Lie algebra $Der(L[V_n])$
is
$$\chi(Der(L[V_n])) = \sum_{1 \leq s} n\cdot d_{s+1}(V_n)\cdot
t^s.$$

\

It seems likely that the analogous series for the Johnson Lie
algebra should admit an analogous description in terms of the
$d_{s+1}(V_n)$.

\

\section{Acknowledgements} \label{sec:Acknowlegements}

The referee made very useful suggestions which clarified this paper;
we thank the referee for thorough, well-done work. The authors thank
Alexandru Suciu and Stefan Papadima for their interest as well as
important corrections. Finally, the authors thank Shigeyuki Morita
for explicating deep features of these groups.

\

The first author was partially supported by DARPA grant number
2006-06918-01. The third author was supported in part by NSF grant
DMS-0856143 and NSF RTG grant DMS-0602191.

%
%

\end{document}